\documentclass[a4paper]{amsart}

\usepackage{amsmath, amsfonts, amsthm , amssymb, amscd}
\usepackage{graphicx}
\usepackage{pst-grad}
\usepackage{url}
\usepackage{pstricks}
\usepackage[utf8]{inputenc}
\usepackage{pstricks-add}
\usepackage{pgf}
\usepackage{tikz}
\usepackage[thinlines]{easytable}

 \newtheorem{thm}{Theorem}[section]
 \newtheorem{prop}[thm]{Proposition}
 
 \newtheorem{cor}[thm]{Corollary}
\theoremstyle{definition}
 \newtheorem{rem}[thm]{Remark}
 \newtheorem{definition}{Definition}[section]
\numberwithin{equation}{section}
\newtheorem{example}{\sc Example}

\newcommand\mj{\mbox{\bf 1}}

\newcommand\set{\mathbf{Set}}
\newcommand\rel{\mathbf{Rel}}
\newcommand\cmd{\mathbf{Cmd}}
\newcommand\relo{\mathbf{Rel}_\omega}
\newcommand\cob{\mathbf{Cob}}
\newcommand\mat{\mathbf{Mat}}
\newcommand\vect{\mathbf{Vect}}
\newcommand\fdvect{\mathbf{fdVect}}
\newcommand\fdhilb{\mathbf{fdHilb}}

\def\d#1{{#1\kern-0.4em\char"16\kern-0.1em}}
\def\D#1{{\raise0.2ex\hbox{-}\kern-0.4em#1}}
\def \Dj{\mbox{\raise0.3ex\hbox{-}\kern-0.4em D}}
\definecolor{britishracinggreen}{rgb}{0.0, 0.26, 0.15}

\title{Coherence for closed categories with biproducts}

\author[Petri\' c]{Zoran Petri\' c}
\address{\scriptsize{Mathematical Institute SANU\\ Knez Mihailova 36, p.f.\ 367\\ 11001 Belgrade, Serbia}}
\email{zpetric@mi.sanu.ac.rs}
\author[Zeki\' c]{Mladen Zeki\' c}
\address{\scriptsize{Mathematical Institute SANU\\ Knez Mihailova 36, p.f.\ 367\\ 11001 Belgrade, Serbia}}
\email{mzekic@mi.sanu.ac.rs}

\date{}

\begin{document}

\begin{abstract}
A coherence result for symmetric monoidal closed categories with biproducts is shown in this paper. It is also explained how to prove coherence for compact closed categories with biproducts and for dagger compact closed categories with dagger biproducts by using the same technique.


\noindent {\small {\it Keywords$\,$}: symmetric monoidal closed category, compact
closed
category, dagger category, enriched category, cobordism}
\end{abstract}

\maketitle

\section{Introduction}

The aim of this paper is to prove the following result:
\begin{quotation}
\emph{The category of 1-dimensional cobordisms, freely enriched over the category of commutative monoids and completed with respect to biproducts, provides a proper graphical language for closed categories with biproducts.}
\end{quotation}
This coherence result is formally stated through Theorems \ref{coh} and
\ref{coh1}-\ref{coh2} below. The first of these theorems treats the case of symmetric monoidal closed categories with biproducts. As in the case of symmetric monoidal closed categories, the commuting diagrams are restricted to those involving ``proper'' objects. This result says that every two canonical arrows from $a$ to $b$ (for $a$ and $b$ proper) with the same ``graphs'' are equal in such a category. However, the notion of the graph of an arrow is somewhat different in this case---it is a matrix whose entries are formal sums of graphs adequate for symmetric monoidal closed case (the Kelly-Mac Lane graphs). The second and the third theorem are analogous. They treat the cases of compact closed categories with biproducts and dagger compact closed categories with dagger biproducts. The main difference is that the latter results are not restricted to proper objects.

Coherence, as a category theoretical notion, finds its roots in the papers of Mac Lane, \cite{ML63}, and Stasheff, \cite{S63}. Since then, lots of coherence results have been proven and possible applications have been found in many fields of mathematics. We mention
just a few appearances of such results in category theory, \cite[XI.3,
Theorem~1]{ML71}, in mathematical linguistics and logic, \cite[Proposition~4]{L68}, in homotopy theory, \cite[Theorem~3.6]{BFSV}, \cite[Theorems 3.1-2]{PT13}, in combinatorics, \cite[Theorem~2.5]{K93}, \cite[Theorem~5.2]{BIP19}, in low-dimensional topology, \cite[Theorem~2.5]{T10} and in mathematical physics,~\cite{S09}.

As one can see from the examples above, coherence results are formulated in many different (sometimes hardly recognisable) forms. The approach to coherence in this paper is the one established in \cite{DP04}, namely, coherence for a category theorist is nothing but completeness for a logician. It stems from Kelly's attempt, \cite[Section~1.4, pp.\ 111-112]{K72}, to make uniform the notion of coherence, which is further
developed by Voreadou, \cite[Introduction, p. viii]{V77}, and Soloviev, \cite{S90}, \cite{S97}. According to this approach, on the side of syntax, we have a freely generated category $\mathcal{C}$ whose language and axiomatic commuting diagrams are specified, while on the side of semantics we have some kind of graphs (a graphical language), which
may be formalised as arrows of a category $\mathcal{D}$ of the ``same type'' as $\mathcal{C}$. Then, following \cite{DP04}, a coherence result may be stated as existence of a faithful functor from $\mathcal{C}$ to $\mathcal{D}$. Since one expects out of such a result a decision procedure for diagram commuting problem, it is desirable to have this
problem decidable in $\mathcal{D}$ (cf.\ the notion of \emph{manageable} category given in \cite[\S 1.1]{DP04}).

%% too many "tied": could it become "related", "connected", "linked",...?
%% or be turned passive? Maybe just the third one.
Traditionally, the graphs associated to arrows of closed categories are based on 1-dimensional manifolds (cf.\ \cite{KML71} and \cite{KL80}), while the graphs adequate for arrows of categories with products, coproducts and biproducts contain branchings (singularities) and hence are not manifolds (cf.\ \cite{S09}). These two graphical languages do not cooperate well, as it was noted in \cite[Section~3, last paragraph]{S07}. The main problem related to this discrepancy is to find a proper graphical language for cartesian closed categories, and it remains open. On the other hand, from the point of view of category theory, the closed structure goes perfectly well with biproducts---the former distributes over the latter. Also, there are lots of examples possessing both structures. However, the only coherence result we know from the literature, which treats closed categories with biproducts, is \cite[Theorem~21]{AD04}.

The structures investigated in this paper are of particular interest for researchers working in quantum information and computation (cf.\ \cite{AC04}, \cite{AD04}, \cite{S07} and \cite{H09}). Our interest for closed categories with biproducts is motivated by questions
arising from categorial proof theory. A recent research, \cite{PS}, in which both authors have participated, considers a sequent system with a connective that acts simultaneously as conjunction and disjunction. From the standpoint of categorial proof theory, such a connective corresponds to a biproduct.

We hope that our results could interact with research concerning the problem of full coherence for closed categories (see \cite{S90}, \cite{S97} and \cite{MS07}), where sometimes (cf.\ \cite[Lemma~2.7]{S97}) the role of biproducts is evident. The
language we cover in this paper includes basic notions used in homological algebra---potentially, our results can simplify some diagram chasing. Also, our approach opens up the possibility to construct other graphical languages for some more involved structures in order to extend a very systematic list given in \cite{S09}.

In the last section of the paper, we mention some open problems. A
possibility to switch from one type of graphs to another, in coherence result for closed categories with biproducts, by using topological quantum field theories seems to be of particular interest.

%We assume that the reader is familiar with basic notions in category theory.

\section{Closed categories and biproducts}

A brief review of some categorial notions relevant for our results is given in this section. A \emph{symmetric monoidal category} is a category $\mathcal{A}$ equipped with a distinguished object $I$, a bifunctor $\otimes\colon \mathcal{A}\times \mathcal{A}\to \mathcal{A}$ and the natural isomorphisms $\alpha$, $\lambda$ and $\sigma$ with components $\alpha_{a,b,c}\colon a\otimes(b\otimes c)\to (a\otimes b)\otimes c$, $\lambda_a\colon I\otimes a\to a$ and $\sigma_{a,b}\colon a\otimes b\to b\otimes a$. Moreover, the coherence conditions concerning the arrows of $\mathcal{A}$ (see the equalities~\ref{19}-\ref{21} below) hold.

A \emph{symmetric monoidal closed category} is a symmetric monoidal category $\mathcal{A}$ in which for every object $a$ there is a right adjoint $a\multimap\colon \mathcal{A}\to \mathcal{A}$ to the functor $a\otimes$. A \emph{compact closed category} is a symmetric monoidal category in which every object $a$ has a \emph{dual} $a^\ast$ in the sense that there are arrows $\eta\colon I\to a^\ast\otimes a$ and $\varepsilon\colon a\otimes a^\ast\to I$ such that
\begin{equation}\label{triang}
(a^\ast\otimes \varepsilon)\circ\alpha^{-1}_{a^\ast,a,a^\ast}\circ (\eta\otimes a^\ast)=\sigma_{I,a^\ast},\quad (\varepsilon\otimes a)\circ\alpha_{a,a^\ast,a}\circ(a\otimes\eta)=\sigma_{a,I}.
\end{equation}
Every compact closed category is symmetric monoidal closed since $a^\ast\otimes$ is a right adjoint to $a\otimes$ for every object $a$ of such a category.

A \emph{dagger category} is a category $\mathcal{A}$ equipped with a functor $\dagger\colon \mathcal{A}^{op}\to \mathcal{A}$ such that for every object $a$ and every arrow $f$ of this category $a^\dagger=a$, and $f^{\dagger\dagger}=f$. (For more details see \cite{S07} and \cite{H09}.) A \emph{dagger compact closed category} is a compact closed category $\mathcal{A}$, which is also a dagger category satisfying

\begin{equation}\label{a}
   (f\otimes g)^\dagger=f^\dagger\otimes g^\dagger,
\end{equation}
\begin{equation}\label{b}
   \alpha_{a,b,c}^\dagger=\alpha^{-1}_{a,b,c},\quad \lambda_a^\dagger=\lambda^{-1}_a,\quad \sigma_{a,b}^\dagger=\sigma_{b,a},
\end{equation}
\begin{equation}\label{c}
   \sigma_{a,a^\ast}\circ\varepsilon^\dagger=\eta.
\end{equation}
This notion was introduced by Abramsky and Coecke, \cite{AC04}, under the name ``strongly compact closed category''. (For the reasons to switch to another terminology see \cite[Remark~2.7]{S07}.)

A \emph{zero object} (or a null object) in a category is an object which is both initial and terminal. If a category contains a zero object $0$, then for every pair $a$, $b$ of its objects, there is a composite $0_{a,b}\colon a\to 0\to b$. (For every other zero object $0'$ of this category, the composite $a\to 0'\to b$ is equal to $0_{a,b}$.) A  \emph{biproduct} of $a_1$ and $a_2$ in a category with a zero object consists of a coproduct and a product diagram
\[
a_1\stackrel{\iota^1\:}{\longrightarrow} a_1\oplus a_2 \stackrel{\:\iota^2}{\longleftarrow}a_2,\quad\quad\quad
a_1\stackrel{\:\pi^1}{\longleftarrow} a_1\oplus a_2 \stackrel{\pi^2\:}{\longrightarrow}a_2
\]
for which
\[
\pi^j\circ\iota^i= \left\{\begin{array}{ll}
    \mj_{a_i}, & i=j,
    \\[1ex]
    0_{a_i,a_j}, & \mbox{\rm otherwise}, \end{array} \right .
\]
where $i,j\in\{1,2\}$ (cf.\ the equalities \ref{13}-\ref{14} below).

More generally, a biproduct of a family of objects $\{a_j\mid j\in J\}$ consists of a universal cocone and a universal cone
\[
\{\iota^j\colon a_j\to B\mid j\in J\}, \quad\quad\quad \{\pi^j\colon B\to a_j\mid j\in J\}
\]
for which the above equality holds for all $i,j\in J$. A \emph{category with biproducts} is a category with zero object and biproducts for every pair of objects. Note that a category with biproducts has biproducts for all finite families of objects, but not necessary for infinite families of objects. A biproduct is a \emph{dagger biproduct} when $\iota^j=(\pi^j)^\dagger$, for every $j\in J$.

By defining $f+g$ for $f,g\colon a\to b$ as $\mu_b\circ(f\oplus g)\circ\bar{\mu}_a$, where $\mu_b\colon b\oplus b\to b$ is the \emph{codiagonal} map, and $\bar{\mu}_a\colon a\to a\oplus a$ is the \emph{diagonal} map tied to the coproduct $b\oplus b$ and to the product $a\oplus a$ one obtains an operation on the set of arrows from $a$ to $b$ which is commutative and has $0_{a,b}$ as neutral. Moreover, the composition distributes over $+$. Hence, every category with biproducts may be conceived as a category enriched over the category $\cmd$ of \emph{commutative monoids}.

\begin{example}
The category $\set$ of sets and functions is symmetric monoidal closed with $\otimes$ being the Cartesian product, and $X\multimap Y$ being the set of functions from $X$ to $Y$. More generally, every \emph{cartesian closed category} is symmetric monoidal closed. Even restricted to finite sets, $\set$ is not compact closed. There are no zero objects and biproducts in this category.
\end{example}

\begin{example}
The category $\set^\ast$ of pointed sets whose objects are sets each of which contains a distinguished element, and whose arrows are functions that preserve the distinguished element is symmetric monoidal closed with $\otimes$ being the \emph{smash} product (all the pairs having at least one component distinguished are identified into the distinguished element of the product) and $X\multimap Y$ being the set of all functions from $X$ to $Y$ that preserve the distinguished element, with the distinguished element being the function that maps each element of $X$ to the distinguished element of $Y$. Note that the smash product is not a product in $\set^\ast$, hence, the above structure is not cartesian closed. Also, this category is not compact closed. Every singleton is a zero object in $\set^\ast$, but this is not a category with biproducts.
\end{example}

\begin{example}
The category $\rel$ of sets and relations is dagger compact closed with dagger biproducts of all families of objects. The bifunctor $\otimes$ is the Cartesian product. For every relation $\rho$ its converse (transpose) is $\rho^\dagger$. Every object is self-dual. The arrow $\eta\colon \{\ast\}\to X\times X$ is the relation $\{(\ast,(x,x))\mid x\in X\}$, while $\varepsilon\colon X\times X\to \{\ast\}$ is its converse $\{((x,x),\ast)\mid x\in X\}$. The biproduct of a family of objects is given by their disjoint union and standard injections, while the converse of an injection is the corresponding projection. The category $\relo$ is the full subcategory of $\rel$ on finite ordinals. This category is also dagger compact closed with dagger biproducts.
\end{example}

\begin{example}
For any field $K$, the category $\vect_K$ of vector spaces over $K$ is symmetric monoidal closed with $\otimes$ being the usual tensor product and $V\multimap W$ being the vector space of linear transformations from $V$ to $W$. The zero object of $\vect_K$ is the trivial vector space and the biproduct of $V$ and $W$ is given by the direct sum $V\oplus W$. There are no biproducts of infinite families of non-zero vector spaces. The full subcategory $\fdvect_K$ of $\vect_K$ on finite dimensional vector spaces is compact closed. The usual dual space $V\multimap K$ plays the role of a dual $V^\ast$ of $V$ in $\fdvect_K$. For $(e_i)_{1\leq i\leq n}$ being a basis of $V$ and $(e^i)_{1\leq i\leq n}$ being its dual basis of $V^\ast$, the linear transformations $\eta\colon K\to V^\ast \otimes V$ and $\varepsilon\colon V\otimes V^\ast \to K$ are determined by
\[
\eta(1)=\sum_{i=1}^n e^i\otimes e_i, \quad\quad\quad \varepsilon(e_i\otimes e^j)=e^j(e_i)= \left\{\begin{array}{ll}
    1, & i=j,
    \\[1ex]
    0, & \mbox{\rm otherwise}. \end{array} \right .
\]
\end{example}

\begin{example}
The category $\fdhilb$ of finite dimensional Hilbert spaces (over $\mathbb{C}$) is dagger compact closed with dagger biproducts. For every arrow $f$ of this category, $f^\dagger$ is its unique adjoint determined by $\langle f(x),y\rangle =\langle x, f^\dagger(y)\rangle$.
\end{example}

\begin{example}\label{ex6}
For any \emph{rig} $(R,+,\cdot,0,1)$ ($(R,+,0)$ commutative monoid, $(R,\cdot,1)$ monoid, plus distributivity $x\cdot(y+z)=(x\cdot y)+(x\cdot z)$, $(y+z)\cdot x=(y\cdot x)+(z\cdot x)$, $0\cdot x=0=x\cdot 0$), consider the category $\mat_R$ whose objects are finite ordinals and arrows from $n$ to $m$ are $m\times n$ matrices over $R$, with matrix multiplication as composition. This category is dagger compact closed with dagger biproducts. The bifunctor $\otimes$ is given by the multiplication on objects and the Kronecker product on arrows. For every matrix $A$ over $R$, its transpose is $A^\dagger$. Every object is self-dual. The arrow $\eta\colon 1\to n\cdot n$ is the column with $n^2$ entries having $1$ at places indexed by $k\cdot n$, $0\leq k<n$ and $0$ at all the other places. The arrow $\varepsilon\colon n\cdot n\to 1$ is the transpose of $\eta$. The biproduct of $n$ and $m$ is given by the sum $n+m$, the injections
\[
\iota^1_{n,m}=\left(\begin{array}{c} E_n \\ 0 \end{array} \right)_{(n+m)\times n},
\quad\quad\quad
\iota^2_{n,m}=\left(\begin{array}{c} 0 \\ E_m \end{array} \right)_{(n+m)\times m},
\]
and the projections $\pi^1_{n,m}=(\iota^1_{n,m})^T$, $\pi^2_{n,m}=(\iota^2_{n,m})^T$.

For the rig $\mathbf{2}=(\{0,1\},+,\cdot,0,1)$ where $1+1=1$, i.e. the Boolean algebra with two elements, the category $\mat_{\mathbf{2}}$ is isomorphic, with respect to dagger compact closed and biproduct structure, to the category $\rel_\omega$. This isomorphism is the identity on objects. For our purposes, the category $\mat_{\mathbb{N}}$ for the rig structure on natural numbers is of particular interest. For any field $K$, the category $\mat_K$ is a skeleton of the category $\fdvect_K$.
\end{example}

\begin{example}\label{ex7}
The category $1\cob$ has as objects the finite sequences of points together with their orientation (either $+$ or $-$). Hence, an object of $1\cob$ is represented by a sequence of $+$ and $-$, e.g.\ $++-+--$. By a 1-\emph{manifold} we mean a compact oriented 1-dimensional topological manifold with boundary (a finite collection of oriented circles and line segments). For $a$, $b$ objects of $1\cob$, a 1-\emph{cobordism} from $a$ to $b$ is a triple $(M,f_0\colon a\to M, f_1\colon b\to M)$, where $M$ is a 1-manifold with boundary $\Sigma_0\coprod \Sigma_1$ whose orientation is induced from the orientation of $M$, the embedding $f_0\colon a\to M$ whose image is $\Sigma_0$ is orientation preserving, while the embedding $f_1\colon b\to M$ whose image is $\Sigma_1$ is orientation reversing. Two cobordisms $(M,f_0,f_1)$ and $(M',f'_0,f'_1)$ from $a$ to $b$ are
\emph{equivalent}, when there is an
orientation preserving homeomorphism $F:M\to M'$ such that the
following diagram commutes.
\begin{center}
\begin{picture}(120,60)

\put(0,30){\makebox(0,0){$a$}}
\put(60,55){\makebox(0,0){$M$}} \put(60,5){\makebox(0,0){$M'$}}
\put(120,30){\makebox(0,0){$b$}}
\put(25,50){\makebox(0,0){$f_0$}}
\put(95,50){\makebox(0,0){$f_1$}}
\put(25,10){\makebox(0,0){$f'_0$}}
\put(95,10){\makebox(0,0){$f'_1$}} \put(67,30){\makebox(0,0){$F$}}

\put(10,35){\vector(2,1){40}} \put(10,25){\vector(2,-1){40}}
\put(110,35){\vector(-2,1){40}} \put(110,25){\vector(-2,-1){40}}
\put(60,45){\vector(0,-1){30}}
\end{picture}
\end{center}
The arrows of $1\cob$ are the equivalence classes of 1-cobordisms. The identity $\mj_a\colon a\to a$ is represented by the cobordism $(a\times I, x\mapsto (x,0), x\mapsto (x,1))$, while $(M,f_0,f_1)\colon a\to b$ and $(N,g_0,g_1)\colon b\to c$ are composed by ``gluing'', i.e.\ by making the pushout of $M\stackrel{f_1}{\longleftarrow} b\stackrel{g_0}{\longrightarrow} N$.

The category $1\cob$ serves to us as a formalisation of Kelly-Mac Lane graphs introduced in \cite{KML71}. Actually, just the arrows of $1\cob$ free of closed 1-manifolds (circles) are sufficient for these matters, and even the orientation is not relevant. However, if one switches from symmetric monoidal closed categories to compact closed categories, the presence of closed components in 1-manifolds is essential (see \cite{KL80}). All the arrows of $1\cob$ are illustrated such that the source of an arrow is at the top, while its target is at the bottom of the picture, hence the direction of pictures is \emph{top to bottom} and not \emph{left to right} (e.g.\ \cite{K03}) or \emph{bottom to top} (e.g.\ \cite{T10}). We omit the orientation of arrows and objects in pictures when this is not essential.

The category $1\cob$ is dagger compact closed. We have (strict) symmetric monoidal structure on $1\cob$ in which $\otimes$ is given by disjoint union, i.e. by putting two cobordisms ``side by side''. Symmetry is generated by transpositions:
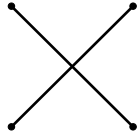
\begin{center}
\begin{tikzpicture}[scale=0.4][line cap=round,line join=round,x=1.0cm,y=1.0cm]
\draw [line width=1pt](0.0,4.0)-- (4.0,0.0);
\draw [line width=1pt](4.0,4.0)-- (0.0,-0.0);
\begin{scriptsize}
\draw [fill=black] (0.0,4.0) circle (3pt);
\draw [fill=black] (4.0,4.0) circle (3pt);
\draw [fill=black] (4.0,0.0) circle (3pt);
\draw [fill=black] (0.0,-0.0) circle (3pt);
\end{scriptsize}
\end{tikzpicture}
\end{center}
(Note that our manifolds are not embedded in the plane and we consider the above cobordism as the disjoint union of two line segments---just the embedding of the source and the target matters.)

The dual $a^{\ast}$ of an object $a$ is the same sequence of points with reversed orientation. For example, if $a=+--$, then $a^{\ast}=-++$. The arrows $\eta\colon \emptyset\to a^\ast\otimes a$ and $\varepsilon\colon a\otimes a^\ast\to \emptyset$, for $a$ as above are the cobordisms illustrated as:
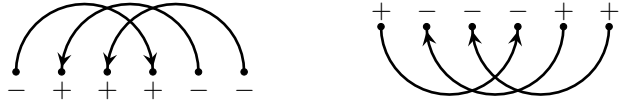
\begin{center}
\begin{tikzpicture}[scale=0.3][line cap=round,line join=round,>=triangle 90,x=1.0cm,y=1.0cm]
\draw [line width=1pt][shift={(-1.0,0.0)}] plot[domain=3.141592653589793:6.283185307179586,variable=\t]({1.0*3.0*cos(\t r)+-0.0*3.0*sin(\t r)},{0.0*3.0*cos(\t r)+1.0*3.0*sin(\t r)});
\draw [line width=1pt][shift={(1.0,0.0)}] plot[domain=3.141592653589793:6.283185307179586,variable=\t]({1.0*3.0*cos(\t r)+-0.0*3.0*sin(\t r)},{0.0*3.0*cos(\t r)+1.0*3.0*sin(\t r)});
\draw [line width=1pt][shift={(3.0,0.0)}] plot[domain=-3.141592653589793:0.0,variable=\t]({1.0*3.0*cos(\t r)+-0.0*3.0*sin(\t r)},{0.0*3.0*cos(\t r)+1.0*3.0*sin(\t r)});
\draw (-4.85,1.5) node[anchor=north west] {$+$};
\draw (3.15,1.5) node[anchor=north west] {$+$};
\draw (5.15,1.5) node[anchor=north west] {$+$};
\draw (-2.85,1.5) node[anchor=north west] {$-$};
\draw (-0.85,1.5) node[anchor=north west] {$-$};
\draw (-20.85,-2) node[anchor=north west] {$-$};
\draw (-18.85,-2) node[anchor=north west] {$+$};
\draw (-16.85,-2) node[anchor=north west] {$+$};
\draw (-14.85,-2) node[anchor=north west] {$+$};
\draw (-12.85,-2) node[anchor=north west] {$-$};
\draw (-10.85,-2) node[anchor=north west] {$-$};
\draw (1.15,1.5) node[anchor=north west] {$-$};
\draw [line width=1pt][shift={(-17.0,-2.0)}] plot[domain=0.0:3.141592653589793,variable=\t]({1.0*3.0*cos(\t r)+-0.0*3.0*sin(\t r)},{0.0*3.0*cos(\t r)+1.0*3.0*sin(\t r)});
\draw [line width=1pt][shift={(-15.0,-2.0)}] plot[domain=0.0:3.141592653589793,variable=\t]({1.0*3.0*cos(\t r)+-0.0*3.0*sin(\t r)},{0.0*3.0*cos(\t r)+1.0*3.0*sin(\t r)});
\draw [line width=1pt][shift={(-13.0,-2.0)}] plot[domain=0.0:3.141592653589793,variable=\t]({1.0*3.0*cos(\t r)+-0.0*3.0*sin(\t r)},{0.0*3.0*cos(\t r)+1.0*3.0*sin(\t r)});
\begin{scriptsize}
\draw [fill=black] (0.0,-0.0) circle (4pt);
\draw [->,>=stealth][line width=1.3pt] (0.04,-0.4) -- (0,-0.2);
\draw [fill=black] (-2.0,0.0) circle (4pt);
\draw [->,>=stealth][line width=1.3pt] (-1.96,-0.4) -- (-2,-0.2);
\draw [fill=black] (-4.0,0.0) circle (4pt);
\draw [fill=black] (2.0,0.0) circle (4pt);
\draw [->,>=stealth][line width=1.3pt] (1.96,-0.4) -- (2,-0.2);
\draw [fill=black] (4.0,0.0) circle (4pt);
\draw [fill=black] (6.0,0.0) circle (4pt);
\draw [fill=black] (-10.0,-2.0) circle (4pt);
\draw [fill=black] (-12.0,-2.0) circle (4pt);
\draw [fill=black] (-14.0,-2.0) circle (4pt);
\draw [->,>=stealth][line width=1.3pt] (-14.04,-1.6) -- (-14,-1.8);
\draw [fill=black] (-16.0,-2.0) circle (4pt);
\draw [->,>=stealth][line width=1.3pt] (-15.96,-1.6) -- (-16,-1.8);
\draw [fill=black] (-18.0,-2.0) circle (4pt);
\draw [->,>=stealth][line width=1.3pt] (-17.96,-1.6) -- (-18,-1.8);
\draw [fill=black] (-20.0,-2.0) circle (4pt);
\end{scriptsize}
\end{tikzpicture}
\end{center}

The equalities~\ref{triang} (in their simplest form, when $a=+$) are illustrated as:
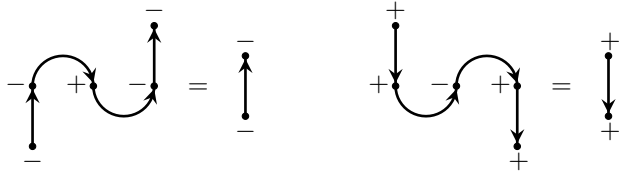
\begin{center}
\begin{tikzpicture}[scale=0.4][line cap=round,line join=round,x=1.0cm,y=1.0cm]
\draw [line width=1pt][shift={(1.0,0.0)}] plot[domain=0.0:3.141592653589793,variable=\t]({1.0*1.0*cos(\t r)+-0.0*1.0*sin(\t r)},{0.0*1.0*cos(\t r)+1.0*1.0*sin(\t r)});
\draw [line width=1pt][shift={(3.0,0.0)}] plot[domain=3.141592653589793:6.283185307179586,variable=\t]({1.0*1.0*cos(\t r)+-0.0*1.0*sin(\t r)},{0.0*1.0*cos(\t r)+1.0*1.0*sin(\t r)});
\draw [line width=1pt](0.0,-0.0)-- (0.0,-2.0);
\draw [line width=1pt](4.0,0.0)-- (4.0,2.0);
\draw [line width=1pt](7.0,1.0)-- (7.0,-1.0);
\draw (4.8,0.4) node[anchor=north west] {$=$};
\draw (3.35,3.1) node[anchor=north west] {$-$};
\draw (6.35,2.1) node[anchor=north west] {$-$};
\draw (6.35,-0.9) node[anchor=north west] {$-$};
\draw (2.8,0.65) node[anchor=north west] {$-$};
\draw (-1.2,0.65) node[anchor=north west] {$-$};
\draw (-0.65,-1.9) node[anchor=north west] {$-$};
\draw (0.8,0.65) node[anchor=north west] {$+$};
\begin{scriptsize}
\draw [fill=black] (0.0,-0.0) circle (3pt);
\draw [->,>=stealth][line width=1.3pt] (0,-0.3) -- (0,-0.1);
\draw [fill=black] (2.0,0.0) circle (3pt);
\draw [->,>=stealth][line width=1.3pt] (1.94,0.3) -- (2,0.1);
\draw [fill=black] (4.0,0.0) circle (3pt);
\draw [->,>=stealth][line width=1.3pt] (3.94,-0.3) -- (4,-0.1);
\draw [fill=black] (0.0,-2.0) circle (3pt);
\draw [fill=black] (4.0,2.0) circle (3pt);
\draw [->,>=stealth][line width=1.3pt] (4,1.7) -- (4,1.9);
\draw [fill=black] (7.0,1.0) circle (3pt);
\draw [->,>=stealth][line width=1.3pt] (7,0.7) -- (7,0.9);
\draw [fill=black] (7.0,-1.0) circle (3pt);
\end{scriptsize}
\end{tikzpicture} \hspace{1cm}
\begin{tikzpicture}[scale=0.4][line cap=round,line join=round,x=1.0cm,y=1.0cm]
\draw [line width=1pt][shift={(11.0,0.0)}] plot[domain=3.141592653589793:6.283185307179586,variable=\t]({1.0*1.0*cos(\t r)+-0.0*1.0*sin(\t r)},{0.0*1.0*cos(\t r)+1.0*1.0*sin(\t r)});
\draw [line width=1pt][shift={(13.0,0.0)}] plot[domain=0.0:3.141592653589793,variable=\t]({1.0*1.0*cos(\t r)+-0.0*1.0*sin(\t r)},{0.0*1.0*cos(\t r)+1.0*1.0*sin(\t r)});
\draw [line width=1pt](10.0,2.0)-- (10.0,0.0);
\draw [line width=1pt](14.0,0.0)-- (14.0,-2.0);
\draw (14.8,0.4) node[anchor=north west] {$=$};
\draw [line width=1pt](17.0,1.0)-- (17.0,-1.0);
\draw (9.35,3.1) node[anchor=north west] {$+$};
\draw (12.8,0.65) node[anchor=north west] {$+$};
\draw (16.4,-0.9) node[anchor=north west] {$+$};
\draw (16.4,2.1) node[anchor=north west] {$+$};
\draw (8.8,0.65) node[anchor=north west] {$+$};
\draw (13.4,-1.9) node[anchor=north west] {$+$};
\draw (10.8,0.65) node[anchor=north west] {$-$};
\begin{scriptsize}
\draw [fill=black] (10.0,0.0) circle (3pt);
\draw [->,>=stealth][line width=1.3pt] (10,0.3) -- (10,0.1);
\draw [fill=black] (12.0,0.0) circle (3pt);
\draw [->,>=stealth][line width=1.3pt] (11.94,-0.3) -- (12,-0.1);
\draw [fill=black] (14.0,0.0) circle (3pt);
\draw [->,>=stealth][line width=1.3pt] (13.94,0.3) -- (14,0.1);
\draw [fill=black] (10.0,2.0) circle (3pt);
\draw [fill=black] (14.0,-2.0) circle (3pt);
\draw [->,>=stealth][line width=1.3pt] (14,-1.7) -- (14,-1.9);
\draw [fill=black] (17.0,1.0) circle (3pt);
\draw [fill=black] (17.0,-1.0) circle (3pt);
\draw [->,>=stealth][line width=1.3pt] (17,-0.7) -- (17,-0.9);
\end{scriptsize}
\end{tikzpicture}
\end{center}

The cobordism $f^\dagger\colon b\to a$ is obtained by reversing the orientation of the 1-manifold representing the cobordism $f\colon a\to b$. (By reversing the orientation of $M$, the embedding $f_0\colon a\to M$ becomes orientation reversing, hence $a$ becomes the target of the obtained cobordism---analogously, $b$ becomes its source.) For example, if $f$ is illustrated at the left-hand side, then $f^\dagger$ is illustrated at the right-hand side of the following picture.
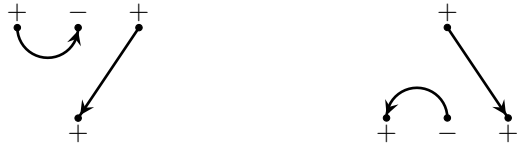
\begin{center}
\begin{tikzpicture}[scale=0.4][line cap=round,line join=round,x=1.0cm,y=1.0cm]
\draw [line width=1pt] (4,4)-- (2,1);
\draw [shift={(1.0,4.0)}, line width=1pt] plot[domain=3.141592653589793:6.283185307179586,variable=\t]({1.0*1.0*cos(\t r)+-0.0*1.0*sin(\t r)},{0.0*1.0*cos(\t r)+1.0*1.0*sin(\t r)});
\draw (1.35,5.1) node[anchor=north west] {$-$};
\draw (-0.65,5.1) node[anchor=north west] {$+$};
\draw (3.35,5.1) node[anchor=north west] {$+$};
\draw (1.35,1.1) node[anchor=north west] {$+$};
\begin{scriptsize}
\draw [fill=black] (0,4) circle (3pt);
\draw [fill=black] (2,4) circle (3pt);
\draw [->,>=stealth][line width=1.3pt] (1.94,3.7) -- (2,3.9);
\draw [fill=black] (4,4) circle (3pt);
\draw [fill=black] (2,1) circle (3pt);
\draw [->,>=stealth][line width=1.3pt] (2.2,1.3) -- (2.067,1.1);
\end{scriptsize}
\end{tikzpicture} \hspace{2.5cm}
\begin{tikzpicture}[scale=0.4][line cap=round,line join=round,x=1.0cm,y=1.0cm]
\draw [line width =1pt][shift={(1.0,0.0)}] plot[domain=0.0:3.141592653589793,variable=\t]({1.0*1.0*cos(\t r)+-0.0*1.0*sin(\t r)},{0.0*1.0*cos(\t r)+1.0*1.0*sin(\t r)});
\draw [line width =1pt] (2.0,3.0)-- (4.0,0.0);
\draw (1.35,4.1) node[anchor=north west] {$+$};
\draw (3.35,0.1) node[anchor=north west] {$+$};
\draw (-0.65,0.1) node[anchor=north west] {$+$};
\draw (1.35,0.1) node[anchor=north west] {$-$};
\begin{scriptsize}
\draw [fill=black] (0,0) circle (3pt);
\draw [->,>=stealth][line width=1.3pt] (0.04,0.2) -- (0,0.1);
\draw [fill=black] (2,0) circle (3pt);
\draw [fill=black] (4,0) circle (3pt);
\draw [->,>=stealth][line width=1.3pt] (3.8,0.3) -- (3.93,0.1);
\draw [fill=black] (2,3) circle (3pt);
\end{scriptsize}
\end{tikzpicture}
\end{center}
It is not hard to check that the equalities~\ref{a}-\ref{c} hold.
\end{example}

\section{SMCB categories}\label{sec3}

This section is devoted to an equational presentation of symmetric monoidal closed categories with biproducts. Our choice of the language, which is very important in such a situation, is the one that provides an easy approach to coherence. A SMCB \emph{category} $\mathcal{A}$ consists of a set of objects and a set of arrows. There are two functions (\emph{source} and \emph{target}) from the set of arrows to the set of objects of $\mathcal{A}$. For every object $a$ of $\mathcal{A}$ there is the identity arrow $\mj_a\colon a\to a$. The set of objects includes two distinguished objects $I$ and $0$. Arrows $f\colon a\to b$ and $g\colon b\to c$ \emph{compose} to give $g\circ f\colon a\to c$, and arrows $f_1,f_2\colon a\to b$ \emph{add} to give $f_1+f_2\colon a\to b$. For every pair of objects $a$ and $b$ of $\mathcal{A}$, there are the objects $a\otimes b$, $a\oplus b$ and $a\multimap b$. Also, for every pair of arrows $f\colon a\to a'$ and $g\colon b\to b'$ there are the arrows $f\otimes g\colon a\otimes b\to a'\otimes b'$, $f\oplus g\colon a\oplus b\to a'\oplus b'$ and $a\multimap g\colon a\multimap b\to a\multimap b'$. In $\mathcal{A}$ we have the following families of arrows indexed by its objects.
\begin{align*}
&\alpha_{a,b,c} \colon a\otimes(b\otimes c)\to (a\otimes b)\otimes c, \quad &&\alpha^{-1}_{a,b,c}\colon (a\otimes b)\otimes c\to a\otimes(b\otimes c),
\\[1ex]
& \lambda_a \colon I\otimes a\to a, \quad &&\lambda^{-1}_a\colon a\to I\otimes a,
\\[1ex]
& \sigma_{a,b}\colon a\otimes b\to b\otimes a,
\\[1ex]
& \eta_{a,b} \colon b\to a\multimap(a\otimes b), && \varepsilon_{a,b} \colon a\otimes(a\multimap b)\to b,
\\[1ex]
& \iota^1_{a,b}\colon a\to a\oplus b,\quad && \iota^2_{a,b}\colon b\to a\oplus b,
\\[1ex]
& \pi^1_{a,b} \colon a\oplus b\to a,\quad && \pi^2_{a,b}\colon a\oplus b\to b,
\\[1ex]
& 0_{a,b} \colon a\to b.
\end{align*}

The arrows of $\mathcal{A}$ should satisfy the following equalities:
\begin{equation}\label{1}
   f\circ \mj_a=f=\mj_{a'}\circ f,\quad (h\circ g)\circ f=h\circ(g\circ f),
\end{equation}
\begin{equation}\label{2}
   \mj_a\otimes \mj_b=\mj_{a\otimes b},\quad (f_2\otimes g_2)\circ(f_1\otimes g_1)= (f_2\circ f_1)\otimes(g_2\circ g_1),
\end{equation}
\begin{equation}\label{3}
   \mj_a\oplus \mj_b=\mj_{a\oplus b},\quad (f_2\oplus g_2)\circ(f_1\oplus g_1)= (f_2\circ f_1)\oplus(g_2\circ g_1),
\end{equation}
\begin{equation}\label{4}
   a\multimap \mj_b=\mj_{a\multimap b},\quad (a\multimap g_2)\circ(a\multimap g_1)= a\multimap(g_2\circ g_1),
\end{equation}
\begin{equation}\label{5}
   \begin{array}{c}((f\otimes g)\otimes h)\circ\alpha_{a,b,c}= \alpha_{a',b',c'}\circ(f\otimes(g\otimes h)),\\[1ex] \alpha^{-1}_{a,b,c}\circ\alpha_{a,b,c}=\mj_{a\otimes(b\otimes c)},\quad \alpha_{a,b,c}\circ\alpha^{-1}_{a,b,c}=\mj_{(a\otimes b)\otimes c},
   \end{array}
\end{equation}
\begin{equation}\label{6}
   f\circ\lambda_a=\lambda_{a'}\circ(I\otimes f),\quad \lambda^{-1}_a\circ\lambda_a=\mj_{I\otimes a},\quad \lambda_a\circ\lambda^{-1}_a=\mj_a,
\end{equation}
\begin{equation}\label{7}
   (g\otimes f)\circ\sigma_{a,b}=\sigma_{a',b'}\circ(f\otimes g),\quad \sigma_{b,a}\circ\sigma_{a,b}=\mj_{a\otimes b},
\end{equation}
\begin{equation}\label{8}
   (a\multimap(a\otimes g))\circ\eta_{a,b}=\eta_{a,b'}\circ g,
\end{equation}
\begin{equation}\label{9}
   g\circ\varepsilon_{a,b}= \varepsilon_{a,b'}\circ (a\otimes(a\multimap g)),
\end{equation}
\begin{equation}\label{10}
   (f\oplus g)\circ\iota^1_{a,b}=\iota^1_{a',b'}\circ f,\quad (f\oplus g)\circ\iota^2_{a,b}=\iota^2_{a',b'}\circ g,
\end{equation}
\begin{equation}\label{11}
   f\circ\pi^1_{a,b}=\pi^1_{a',b'}\circ (f\oplus g),\quad g\circ\pi^2_{a,b}=\pi^2_{a',b'}\circ (f\oplus g),
\end{equation}
\begin{equation}\label{12}
   (a\multimap\varepsilon_{a,b})\circ\eta_{a,a\multimap b}=\mj_{a\multimap b}, \quad \varepsilon_{a,a\otimes b}\circ(a\otimes\eta_{a,b})=\mj_{a\otimes b},
\end{equation}
\begin{equation}\label{13}
   \pi^1_{a,b}\circ\iota^1_{a,b}=\mj_a,\quad \pi^2_{a,b}\circ\iota^2_{a,b}=\mj_b,
\end{equation}
\begin{equation}\label{14}
   \pi^2_{a,b}\circ\iota^1_{a,b}=0_{a,b},\quad \pi^1_{a,b}\circ\iota^2_{a,b}=0_{b,a},
\end{equation}
\begin{equation}\label{15}
   \iota^1_{a,b}\circ \pi^1_{a,b} + \iota^2_{a,b}\circ \pi^2_{a,b}=\mj_{a\oplus b},
\end{equation}
\begin{equation}\label{16}
   f_1+(f_2+f_3)=(f_1+f_2)+f_3,\quad f_1+f_2=f_2+f_1,\quad f+0_{a,a'}=f,
\end{equation}
\begin{equation}\label{17}
   (g_1+g_2)\circ f=g_1\circ f + g_2\circ f,\quad g\circ(f_1+f_2)=g\circ f_1 + g\circ f_2,
\end{equation}
\begin{equation}\label{18}
   0_{a',b}\circ f=0_{a,b},\quad f\circ 0_{b,a}=0_{b,a'}.
\end{equation}
\begin{equation}\label{19}
   \alpha_{a\otimes b,c,d}\circ\alpha_{a,b,c\otimes d}=(\alpha_{a,b,c}\otimes d)\circ\alpha_{a,b\otimes c,d}\circ (a\otimes\alpha_{b,c,d}),
\end{equation}
\begin{equation}\label{20}
   \lambda_{a\otimes b}=(\lambda_a\otimes b)\circ\alpha_{I,a,b},
\end{equation}
\begin{equation}\label{21}
   \alpha_{c,a,b}\circ\sigma_{a\otimes b,c}\circ \alpha_{a,b,c}= (\sigma_{a,c}\otimes b)\circ \alpha_{a,c,b} \circ (a\otimes \sigma_{b,c}),
\end{equation}
\begin{equation}\label{22}
   0_{0,0}=\mj_0.
\end{equation}

The equalities \ref{1} say that $\mathcal{A}$ is a category. The equalities \ref{2}-\ref{4} say that $\otimes$ and $\oplus$ are bifunctors, while $a\multimap$ is a functor. The equalities \ref{5}-\ref{7} say that $\alpha$, $\lambda$ and $\sigma$ are natural isomorphisms. The equalities \ref{8}-\ref{11} say that $\eta_a$, $\varepsilon_a$, $\iota$ and $\pi$ are natural. The equalities \ref{12} are \emph{triangular} equalities. The equalities \ref{13}-\ref{15} are \emph{biproduct} equalities, while the equalities \ref{16}-\ref{18} say that $\mathcal{A}$ is enriched over the category \textbf{Cmd}. The coherence conditions are contained in \ref{19}-\ref{22}.

The equalities \ref{1}, \ref{2}, \ref{5}-\ref{7}, \ref{19}-\ref{21} say that $\mathcal{A}$ is a \emph{symmetric monoidal} category. From \ref{2}, \ref{4}, \ref{8}-\ref{9}, \ref{12}, with the help of \cite[IV.1, Theorem~2(v)]{ML71}, it follows that for every $a$, the functor $a\multimap$ is a right adjoint to the functor $a\otimes$, hence $\mathcal{A}$ is \emph{symmetric monoidal closed}.

Since for every object $a$ of $\mathcal{A}$ the arrows $0_{0,a}$ and $0_{a,0}$ exist, with the help of \ref{18} and \ref{22}, one may conclude that $0$ is a zero object, i.e.\ an initial and a terminal object of $\mathcal{A}$. The following proposition, together with \ref{13}-\ref{14} shows that $\mathcal{A}$ is equipped with biproducts.

\begin{prop}
For every $a$ and $b$,
\[
a\stackrel{\iota^1_{a,b}\:}{\longrightarrow} a\oplus b \stackrel{\:\iota^2_{a,b}}{\longleftarrow}b,\quad\quad\quad
a\stackrel{\:\pi^1_{a,b}}{\longleftarrow} a\oplus b \stackrel{\pi^2_{a,b}\:}{\longrightarrow}b
\]
are coproduct and product diagrams in $\mathcal{A}$, respectively.
\end{prop}
\begin{proof}
For $f\colon a\to c$ and $g\colon b\to c$, the unique arrow $h\colon a\oplus b\to c$ such that $h\circ\iota^1=f$ and $h\circ\iota^2=g$ is obtained as $f\circ\pi^1+g\circ\pi^2$. Dually, for $f\colon c\to a$ and $g\colon c\to b$, the unique arrow $h\colon c\to a\oplus b$ such that $\pi^1\circ h=f$ and $\pi^2\circ h=g$ is obtained as $\iota^1\circ f +\iota^2\circ g$. Note that the uniqueness of $h$ in both cases follows from \ref{15}, and also this equality is necessary for the uniqueness of $h$ in either case, e.g.\ it follows from the uniqueness of $h$ in the first case when we instantiate $f$ by $\iota^1_{a,b}$ and $g$ by~$\iota^2_{a,b}$.
\end{proof}
Hence, every SMCB category is symmetric monoidal closed with biproducts. On the other hand, it is straightforward to check that every symmetric monoidal closed category with biproducts has the SMCB structure.

By defining $f\multimap b\colon a'\multimap b\to a\multimap b$, for $f\colon a\to a'$, as
\[
(a\multimap\varepsilon_{a',b})\circ (a\multimap(f\otimes(a'\multimap b)))\circ \eta_{a,a'\multimap b}
\]
one obtains a bifunctor $\multimap\colon \mathcal{A}^{op}\times \mathcal{A}\to \mathcal{A}$ (see \cite[IV.7, Theorem~3]{ML71}). In this way, $\eta$ and $\varepsilon$ become dinatural, i.e.\ the following two equalities hold.
\begin{equation}\label{23}
   (a\multimap(f\otimes b))\circ\eta_{a,b}=(f\multimap(a'\otimes b))\circ \eta_{a',b},
\end{equation}
\begin{equation}\label{24}
   \varepsilon_{a,b}\circ(a\otimes(f\multimap b))=\varepsilon_{a',b}\circ(f\otimes(a'\multimap b)).
\end{equation}

By \cite[V.5, Theorem~1]{ML71} and its dual we have the following results.
\begin{prop}
For every $a$, $b$ and $c$,
\[
c\multimap a\stackrel{\:c\multimap\pi^1_{a,b}}{\xleftarrow{\hspace*{2.4em}}} c\multimap(a\oplus b) \stackrel{c\multimap\pi^2_{a,b}\:}{\xrightarrow{\hspace*{2.4em}}}c\multimap b,\quad\quad c\otimes a\stackrel{c\otimes\iota^1_{a,b}\:}{\xrightarrow{\hspace*{2em}}} c\otimes(a\oplus b) \stackrel{\:c\otimes\iota^2_{a,b}}{\xleftarrow{\hspace*{2em}}}c\otimes b
\]
are product and coproduct diagrams in $\mathcal{A}$, respectively, while $c\multimap 0$ and $c\otimes 0$ are zero objects.
\end{prop}
\begin{cor}
For every $a$, $b$ and $c$,
\[
c\multimap(a\oplus b)\cong (c\multimap a)\oplus(c\multimap b),\quad c\otimes(a\oplus b)\cong (c\otimes a)\oplus(c\otimes b),\quad c\multimap 0\cong 0\cong c\otimes 0.
\]
\end{cor}

With the help of the above isomorphisms, one derives the following equalities.
\begin{equation*}\label{25}
   f\otimes(g_1+g_2)=(f\otimes g_1)+(f\otimes g_2),\quad (f_1+f_2)\otimes g=(f_1\otimes g)+(f_2\otimes g),
\end{equation*}
\begin{equation*}\label{26}
   f\multimap(g_1+g_2)=(f\multimap g_1)+(f\multimap g_2),\quad (f_1+f_2)\multimap g=(f_1\multimap g)+(f_2\multimap g),
\end{equation*}
\begin{equation*}\label{27}
   f\otimes 0_{b,b'}=0_{a\otimes b,a'\otimes b'}=0_{a,a'}\otimes g,\quad
\end{equation*}
\begin{equation*}\label{28}
   f\multimap 0_{b,b'}=0_{a'\multimap b,a\multimap b'}=0_{a,a'}\multimap g.
\end{equation*}

\section{A free SMCB category}\label{free SMCB}

Our presentation of SMCB categories is purely equational. This enables one to construct a SMCB category $\mathcal{F}_P$ freely generated by an (infinite) set $P$. The \emph{objects} of $\mathcal{F}_P$ are the formulae built out of elements of $P$ and the constants $I$ and $0$, with the help of three binary connectives $\otimes$, $\oplus$ and $\multimap$. In order to obtain the arrows of $\mathcal{F}_P$, we start with \emph{primitive terms} which are of the form $\mj_a$, $\alpha_{a,b,c}$, $\lambda_a$, $\sigma_{a,b}$, $\eta_{a,b}$, $\varepsilon_{a,b}$, $\iota^i_{a,b}$, $\pi^i_{a,b}$ and $0_{a,b}$, for all objects $a$, $b$ and $c$ of  $\mathcal{F}_P$. The \emph{terms} are built out of primitive terms with the help of operational symbols $\otimes$, $\oplus$, $a\multimap$, for every object $a$ of $\mathcal{F}_P$, $+$ and $\circ$. (Each such term is equipped with the source and the target, which are objects of $\mathcal{F}_P$, and constructions of terms with $+$ and $\circ$ are restricted to appropriate sources and targets.) These terms are quotient by the congruence generated by the equalities \ref{1}-\ref{22}. Hence, an arrow of $\mathcal{F}_P$ is the equivalence class of a term.

Let \textbf{Smcb} be the category whose objects are SMCB categories and whose arrows are functors strictly preserving the SMCB structure. The forgetful functor from \textbf{Smcb} to the category $\set$ of sets and functions, which maps a SMCB category to the set of its objects, has a left adjoint, the ``free'' functor $F$. Our category $\mathcal{F}_P$ is the image $FP$ of the set $P$ under the functor $F$.

\begin{definition}\label{proj-inj}
By induction on the complexity of an object $a$ of $\mathcal{F}_P$, we define two finite sequences ${\rm I}_a=\langle\iota^0_a,\ldots,\iota^{n-1}_a\rangle$ and $\Pi_a=\langle\pi^0_a,\ldots,\pi^{n-1}_a\rangle$ of arrows of $\mathcal{F}_P$ in the following way. If $a$ is an element of $P$ or either the constant $I$ or $0$, then $n=1$ and ${\rm I}_a=\langle\mj_a\rangle=\Pi_a$. Let us assume that ${\rm I}_{a_1}=\langle\iota^0_1,\ldots,\iota^{n_1-1}_1\rangle$, $\Pi_{a_1}=\langle\pi^0_1,\ldots,\pi^{n_1-1}_1\rangle$ and ${\rm I}_{a_2}=\langle\iota^0_2,\ldots,\iota^{n_2-1}_2\rangle$, $\Pi_{a_2}=\langle\pi^0_2,\ldots,\pi^{n_2-1}_2\rangle$ are already defined.
\begin{itemize}
\item[$\otimes$] If $a=a_1\otimes a_2$, then $n=n_1\cdot n_2$, and for $0\leq i< n_1\cdot n_2$,
    \[
    \iota^i_a=\iota^{\lfloor i/n_2\rfloor}_1\otimes \iota^{i\text{ mod}\, n_2}_2,\quad\quad \pi^i_a=\pi^{\lfloor i/n_2\rfloor}_1\otimes \pi^{i\text{ mod}\, n_2}_2.
    \]
\item[$\multimap$] If $a=a_1\multimap a_2$, then $n=n_1\cdot n_2$, and for $0\leq i< n_1\cdot n_2$,
    \[
    \iota^i_a=\pi^{\lfloor i/n_2\rfloor}_1\multimap \iota^{i\text{ mod}\, n_2}_2,\quad\quad \pi^i_a=\iota^{\lfloor i/n_2\rfloor}_1\multimap \pi^{i\text{ mod}\, n_2}_2.
    \]
\item[$\oplus$] If $a=a_1\oplus a_2$, then $n=n_1+ n_2$, and for $0\leq i< n_1+ n_2$,
    \[ \iota^i_a= \left\{\begin{array}{ll}
    \iota^1_{a_1,a_2}\circ \iota^i_1, & 0\leq i<n_1,
    \\[1ex]
    \iota^2_{a_1,a_2}\circ \iota^{i-n_1}_2, & \mbox{\rm otherwise}, \end{array} \right . \quad\quad
    \pi^i_a= \left\{\begin{array}{ll}
    \pi^i_1\circ \pi^1_{a_1,a_2}, & 0\leq i<n_1,
    \\[1ex]
    \pi^{i-n_1}_2\circ \pi^2_{a_1,a_2}, & \mbox{\rm otherwise}. \end{array} \right .
\]
\end{itemize}
\end{definition}

\begin{rem}\label{oplus}
Note that when $a=a_1\oplus a_2$ we have that
\[
\iota^i_a=\iota^{1+s_i}_{a_1,a_2}\circ\iota^{i-n_1\cdot s_i}_{1+s_i},\quad\quad \pi^i_a= \pi^{i-n_1\cdot s_i}_{1+s_i}\circ \pi^{1+s_i}_{a_1,a_2},\quad {\rm where}\; s_i=\left\lfloor\frac{\min\{i, n_1\}}{n_1}\right\rfloor.
\]
\end{rem}

\begin{example} If $n_1=3$ and $n_2=2$, then
\begin{alignat*}{2}
& {\rm I}_{a_1\otimes a_2} &={}& \langle \iota^0_1\otimes\iota^0_2, \iota^0_1\otimes\iota^1_2, \iota^1_1\otimes\iota^0_2, \iota^1_1\otimes\iota^1_2, \iota^2_1\otimes\iota^0_2, \iota^2_1\otimes\iota^1_2\rangle,
\\[2ex]
& \Pi_{a_1\otimes a_2} &={}& \langle \pi^0_1\otimes\pi^0_2, \pi^0_1\otimes\pi^1_2, \pi^1_1\otimes\pi^0_2, \pi^1_1\otimes\pi^1_2, \pi^2_1\otimes\pi^0_2, \pi^2_1\otimes\pi^1_2\rangle,
\\[2ex]
& {\rm I}_{a_1\multimap a_2}  &={}& \langle \pi^0_1\multimap\iota^0_2, \pi^0_1\multimap\iota^1_2, \pi^1_1\multimap\iota^0_2, \pi^1_1\multimap\iota^1_2, \pi^2_1\multimap\iota^0_2, \pi^2_1\multimap\iota^1_2\rangle,
\\[2ex]
& \Pi_{a_1\multimap a_2}  &={}& \langle \iota^0_1\multimap\pi^0_2, \iota^0_1\multimap\pi^1_2, \iota^1_1\multimap\pi^0_2, \iota^1_1\multimap\pi^1_2, \iota^2_1\multimap\pi^0_2, \iota^2_1\multimap\pi^1_2\rangle,
\\[2ex]
& {\rm I}_{a_1\oplus a_2}  &={}& \langle \iota^1_{a_1,a_2}\circ\iota^0_1, \iota^1_{a_1,a_2}\circ\iota^1_1, \iota^1_{a_1,a_2}\circ\iota^2_1, \iota^2_{a_1,a_2}\circ\iota^0_2, \iota^2_{a_1,a_2}\circ\iota^1_2\rangle,
\\[2ex]
& \Pi_{a_1\oplus a_2}  &={}& \langle \pi^0_1\circ\pi^1_{a_1,a_2}, \pi^1_1\circ\pi^1_{a_1,a_2}, \pi^2_1\circ\pi^1_{a_1,a_2}, \pi^0_2\circ\pi^2_{a_1,a_2}, \pi^1_2\circ\pi^2_{a_1,a_2}\rangle.
\end{alignat*}
\end{example}

\begin{example} Let $x=(a\oplus b)\oplus c$ and $y=((a\oplus b)\oplus c)\otimes (c\oplus d)$, where $a,b,c,d$ are elements of $P$. Then $\iota_x^i$ for $0\leq i<3$ and $\iota_y^j$ for $0\leq j<6$ are given in the following tables.
\begin{table}[h]
\begin{minipage}{0.49\linewidth}\centering
\begin{TAB}[3pt]{|c|c|}{|c|c|c|}
$\iota_{x}^0$  & $\iota_{a\oplus b,c}^1\circ \iota^1_{a,b}$ \\
$\iota_{x}^1$  & $\iota_{a\oplus b,c}^1\circ \iota^2_{a,b}$ \\
$\iota_{x}^2$  & $\iota_{a\oplus b,c}^2$ \\
\end{TAB}
\end{minipage}
\begin{minipage}{0.49\linewidth}\centering
\begin{TAB}[3pt]{|c|c|}{|c|c|c|c|c|c|}
$\iota_{y}^0$ & $(\iota_{a\oplus b,c}^1\circ \iota^1_{a,b})\otimes \iota_{c,d}^1$ \\
$\iota_{y}^1$ & $(\iota_{a\oplus b,c}^1\circ \iota^1_{a,b})\otimes \iota_{c,d}^2$ \\
$\iota_{y}^2$ & $(\iota_{a\oplus b,c}^1\circ \iota^2_{a,b})\otimes \iota_{c,d}^1$ \\
$\iota_{y}^3$ & $(\iota_{a\oplus b,c}^1\circ \iota^2_{a,b})\otimes \iota_{c,d}^2$ \\
$\iota_{y}^4$ & $\iota^2_{a\oplus b,c}\otimes \iota^1_{c,d}$ \\
$\iota_{y}^5$ & $\iota^2_{a\oplus b,c}\otimes \iota^2_{c,d}$ \\
\end{TAB}
\end{minipage}
\end{table}
\end{example}

\begin{rem}\label{3.1}
When $a$ is built out of elements of $P$ using only $\oplus$, the sequence $\mathrm{I}_a$ ($\Pi_a$) consists of all the injections (projections) of the atoms of $a$, while this is not true when $a$ contains $\oplus$ in the scope of $\otimes$ (or $\multimap$). For every $0\leq i<n$, the target of $\iota^i_a$ and the source of $\pi^i_a$ are both equal to $a$, while the source $a^i$ of $\iota^i_a$ is equal to the target of $\pi^i_a$, and $a^i$ is $\oplus$-free. Moreover, if $a$ is $\oplus$-free, then ${\rm I}_a=\langle\mj_a\rangle=\Pi_a$.
\end{rem}

The following proposition has a straightforward proof.

\begin{prop}\label{4.2}
For every object $a$ of $\mathcal{F}_P$
\[
\pi^j_a\circ\iota^i_a= \left\{
\begin{array}{ll} \mj_{a^i}, & i=j,\\[1ex]
0_{a^i,a^j}, & \mbox{\rm otherwise},
\end{array} \right . \quad\quad\quad \sum_{i=0}^{n-1} \iota^i_a\circ\pi^i_a =\mj_a.
\]
\end{prop}

\begin{cor}\label{3.3}
For every object $a$ of $\mathcal{F}_P$, the cocone $(a,{\rm I}_a)$ together with the cone $(a,\Pi_a)$ make a biproduct.
\end{cor}

\section{A matrix normalisation}\label{matrix}

Our next goal is to eliminate $\oplus$, $\iota$ and $\pi$ from every arrow of $\mathcal{F}_P$, whose source and target are $\oplus$-free. The following \emph{matrix normalisation} of terms provides a solution. An alternative solution could be obtained via procedure akin to Kleene's permutation of inference rules (see \cite{K52}). Namely, one could define a correspondence between $\iota$'s and $\pi$'s in arrows whose source and target are $\oplus$-free, and then bring, by permutations based on naturality and functoriality, a corresponding pair together, in order to be eliminated. However, we find the following procedure more elegant.

For every arrow $u\colon a\to b$ of $\mathcal{F}_P$, where ${\rm I}_a= \langle\iota^0_a, \ldots,\iota^{n-1}_a\rangle$, $\Pi_b=\langle\pi^0_b,\ldots,\pi^{m-1}_b\rangle$, let $M_u$ be the $m\times n$ matrix whose $ij$ entry is $\pi^i_b\circ u\circ \iota^j_a$.
Let $X_{m_1\times n_1}$ and $Y_{m_2\times n_2}$ be two matrices of arrows of $\mathcal{F}_P$. For $\bullet$ being $\otimes$ or $\multimap$, following the definition of the Kronecker product of matrices, let $K_\bullet(X,Y)$ be the $(m_1\cdot m_2) \times (n_1\cdot n_2)$ matrix whose $ij$ entry is
\[
x_{\lfloor i/m_2\rfloor,\lfloor j/n_2\rfloor} \bullet y_{i\text{ mod}\, m_2,j\text{ mod}\, n_2}.
\]
For example,
\[
K_\bullet\left(\left(\begin{array}{ccc} x_{00} & x_{01} & x_{02}
\\
x_{10} & x_{11} & x_{12} \end{array}\right), \left(\begin{array}{cc} y_{00} & y_{01}
\\
y_{10} & y_{11} \end{array}\right)\right)
\]
is
\[
\left(\begin{array}{cccccc}
x_{00}\bullet y_{00} & x_{00}\bullet y_{01} & x_{01}\bullet y_{00} & x_{01}\bullet y_{01} & x_{02}\bullet y_{00} & x_{02}\bullet y_{01}
\\
x_{00}\bullet y_{10} & x_{00}\bullet y_{11} & x_{01}\bullet y_{10} & x_{01}\bullet y_{11} & x_{02}\bullet y_{10} & x_{02}\bullet y_{11}
\\
x_{10}\bullet y_{00} & x_{10}\bullet y_{01} & x_{11}\bullet y_{00} & x_{11}\bullet y_{01} & x_{12}\bullet y_{00} & x_{12}\bullet y_{01}
\\
x_{10}\bullet y_{10} & x_{10}\bullet y_{11} & x_{11}\bullet y_{10} & x_{11}\bullet y_{11} & x_{12}\bullet y_{10} & x_{12}\bullet y_{11}
\end{array} \right).
\]

For two such matrices $X$ and $Y$, we define
\[
X\otimes Y=_{df}K_\otimes (X,Y),\quad\quad\quad X\multimap Y=_{df} K_\multimap(X^T,Y),
\]
while $X\oplus Y$ is the direct sum
\[
\left(\begin{array}{cc} X & 0 \\ 0 & Y\end{array}\right)
\]
of $X$ and $Y$. If $X$ and $Y$ are of the same type having the corresponding elements in the same hom-sets, then $X+Y$ is the matrix of the same type whose $ij$ entry is $x_{ij}+y_{ij}$. If $X_{m\times p}$ and $Y_{p\times n}$, and for every $0\leq i<m$, $0\leq j<n$ the compositions $x_{ik}\circ y_{kj}$ are defined for every $0\leq k<p$, and belong to the same hom-set, then we define $X\circ Y$ as the $m\times n$ matrix whose $ij$ entry is $\sum_{k=0}^{p-1} x_{ik}\circ y_{kj}$.

\begin{prop}\label{matrix1}
For $\bullet$ being $\otimes$, $\multimap$, $\oplus$, $+$ and $\circ$, we have
\[
M_{u_1\bullet u_2}=M_{u_1}\bullet M_{u_2}.
\]
\end{prop}
\begin{proof} For the first three cases below, let us assume that $u_i\colon a_i\to b_i$ and that $M_{u_i}$ is an $m_i\times n_i$ matrix, where $i\in\{1,2\}$.

\medskip
\noindent (1) If $\bullet$ is $\otimes$, then we have
\begin{align*}
(M_{u_1\otimes u_2})_{i,j} &= \pi_{b_1\otimes b_2}^i\circ (u_1\otimes u_2)\circ \iota_{a_1\otimes a_2}^j \\
&=(\pi_1^{\lfloor \frac{i}{m_2}\rfloor}\otimes \pi_2^{i\;\mathrm{mod}\;m_2}) \circ (u_1\otimes u_2) \circ (\iota_1^{\lfloor \frac{j}{n_2}\rfloor}\otimes \iota_2^{j\;\mathrm{mod}\;n_2}) \\
&= (\pi_1^{\lfloor \frac{i}{m_2}\rfloor} \circ u_1\circ \iota_1^{\lfloor \frac{j}{n_2}\rfloor})\otimes (\pi_2^{i\;\mathrm{mod}\;m_2}\circ u_2\circ \iota_2^{j\;\mathrm{mod}\;n_2}) \\
&= (M_{u_1}\otimes M_{u_2})_{i,j}.
\end{align*}

\medskip
\noindent (2) We proceed analogously when $\bullet$ is $\multimap$.
(Note that, for the sake of Corollary~\ref{4.3}, since $a_1\multimap u_2=\mj_{a_1}\multimap u_2$, it suffices here to consider just the case when $u_1$ is $\mj_{a_1}$.)

\medskip
\noindent (3) If $\bullet$ is $\oplus$, then, by relying on Remark~\ref{oplus}, we have
\begin{align*}
(M_{u_1\oplus u_2})_{i,j} &= \pi_{1+s_i}^{i-m_1\cdot s_i}\circ \pi_{b_1,b_2}^{1+s_i} \circ (u_1\oplus u_2)\circ \iota_{a_1,a_2}^{1+s_j}\circ \iota_{1+s_j}^{j-n_1\cdot s_j} \\
&= \pi_{1+s_i}^{i-m_1\cdot s_i} \circ u_{1+s_i}\circ \pi_{a_1,a_2}^{1+s_i} \circ \iota_{a_1,a_2}^{1+s_j}\circ \iota_{1+s_j}^{j-n_1\cdot s_j} \\
&= \begin{cases}
\pi_1^{i}\circ u_1\circ \iota_1^{j}, & 0\leq i< m_1, \; 0\leq j<n_1, \\
\pi_2^{i-m_1}\circ u_2\circ \iota_2^{j-n_1}, & m_1\leq i<m_1+m_2, \; n_1\leq j<n_1+n_2, \\
0, & \text{otherwise},
\end{cases} \\
&=(M_{u_1}\oplus M_{u_2})_{i,j}.
\end{align*}

\medskip
\noindent (4) If $\bullet$ is $+$, and $u_1,u_2\colon a\to b$, then we have
\begin{align*}
(M_{u_1+u_2})_{i,j} &= \pi_b^i\circ (u_1+u_2)\circ \iota_a^j \\
&= \pi_b^i\circ u_1\circ \iota_a^j + \pi_b^i\circ u_2\circ \iota_a^j \\
&= (M_{u_1}+M_{u_2})_{i,j}.
\end{align*}

\medskip
\noindent (5) If $\bullet$ is $\circ$, and $u_1:b\to c$, $u_2:a\to b$, while $M_{u_1}$ is a $k\times m$ and $M_{u_2}$ is an $m\times n$ matrix, then, by relying on Proposition~\ref{4.2}, we have
\begin{align*}
(M_{u_1}\circ M_{u_2})_{i,j}&=\sum_{l=0}^{m-1} \pi_c^i\circ u_1\circ \iota_b^{l}\circ \pi_{b}^{l}\circ u_2\circ \iota_a^j \\
&= \pi_c^i\circ u_1\circ \left[\sum_{l=0}^{m-1} \iota_b^{l}\circ \pi_{b}^{l}\right] \circ u_2\circ \iota_a^j \\
&= \pi_c^i\circ u_1\circ u_2\circ \iota_a^j
= (M_{u_1\circ u_2})_{i,j}.
\end{align*}

\vspace{-1.7em}
\end{proof}

\begin{prop}\label{matrix2}
If $u$ is of the form $\mj_a$, $\alpha_{a,b,c}$, $\lambda_a$, $\sigma_{a,b}$, $\eta_{a,b}$, $\varepsilon_{a,b}$, $\iota^i_{a,b}$, $\pi^i_{a,b}$ or $0_{a,b}$, then all the entries of the matrix $M_u$ are of the form $\mj_p$, $\alpha_{p,q,r}$, $\lambda_p$, $\sigma_{p,q}$, $\eta_{p,q}$, $\varepsilon_{p,q}$ and $0_{p,q}$, where $p$ and $q$ are $\oplus$-free.
\end{prop}
\begin{proof} (1) If $u$ is $\mj_a$, then the $ij$ entry of the matrix $M_u$ is
\[
(M_u)_{i,j}=\pi_a^i\circ \mj_a \circ \iota_a^j=\begin{cases}
\mathbf{1}_{a^i}, & i=j,  \\
0_{a^j, a^i}, & \text{otherwise}.
\end{cases}
\]

\medskip
\noindent (2) If $u$ is $\alpha_{a,b,c}$, then for some $i_1, i_2, i_3$ and $j_1, j_2, j_3$
\begin{align*}
(M_u)_{i,j}&=\pi_{(a\otimes b)\otimes c}^i\circ \alpha_{a,b,c}\circ \iota_{a\otimes (b\otimes c)}^j \\ &= ((\pi_{a}^{i_1}\otimes \pi_{b}^{i_2})\otimes \pi_{c}^{i_3}) \circ \alpha_{a,b,c} \circ (\iota_{a}^{j_1}\otimes (\iota_{b}^{j_2}\otimes \iota_{c}^{j_3})) \\
&= \begin{cases}
\alpha_{a^{i_1},b^{i_2}, c^{i_3}}, & i_1=j_1, \; i_2=j_2, \; i_3=j_3, \\
0_{a^{j_1}\otimes(b^{j_2}\otimes c^{j_3}), (a^{i_1}\otimes b^{i_2})\otimes c^{i_3}}, & \text{otherwise}.
\end{cases}
\end{align*}

\medskip
\noindent (3) We proceed analogously when $u$ is $\lambda_a$ or $\sigma_{a,b}$.

\medskip
\noindent (4) If $u$ is $\eta_{a,b}$, then for some $i_1$, $i_2$, $i_3$, by using \ref{8} and \ref{23} we have
\begin{align*}
(M_u)_{i,j} &= \pi_{a\multimap(a\otimes b)}^i\circ \eta_{a,b}\circ \iota_b^j=(\iota_a^{i_1}\multimap (\pi_a^{i_2}\otimes \pi_{b}^{i_3})) \circ \eta_{a,b}\circ {\iota_b^{j}} \\
&= ((\pi_a^{i_2}\circ \iota_a^{i_1})\multimap (a^{i_2}\otimes (\pi_b^{i_3}\circ \iota_b^{j}))) \circ \eta_{a^{i_2}, b^{j}} \\
&=\begin{cases}
\eta_{a^{i_1}, b^{j}}, & i_1=i_2,\; i_3=j, \\
0_{b^j, a^{i_1}\multimap (a^{i_2}\otimes b^{i_3})}, & \text{otherwise}.
\end{cases}
\end{align*}

\medskip
\noindent (5) We proceed analogously when $u$ is $\varepsilon_{a,b}$.

\medskip
\noindent (6) If $u$ is $\iota_{a,b}^{1}$, then
$(M_u)_{i,j}=\pi_{a\oplus b}^i\circ \iota_{a,b}^{1} \circ \iota_{a}^{j}$, which is either $\pi^{i_1}_a\circ \pi^1_{a,b} \circ \iota_{a,b}^{1}\circ \iota_{a}^{j}$ for some $i_1$, or
$\pi_{b}^{i_2}\circ \pi^2_{a,b} \circ \iota_{a,b}^{1}\circ \iota_{a}^{j}$, for some $i_2$. Moreover,
\[
\pi^{i_1}_a\circ \pi^1_{a,b} \circ \iota_{a,b}^{1}\circ \iota_{a}^{j}=
\begin{cases}
\mathbf{1}_{a^j}, & j=i_1, \\
0_{a^j, a^{i_1}}, & \text{otherwise},
\end{cases}
\quad\quad
\pi_{b}^{i_2}\circ \pi^2_{a,b} \circ \iota_{a,b}^{1}\circ \iota_{a}^{j}=
0_{a^j,b^{i_2}}.
\]

\medskip
\noindent (7) We proceed analogously when $u$ is $\iota_{a,b}^{2}$, $\pi_{a,b}^1$ or $\pi_{a,b}^2$.

\medskip
\noindent (8) If $u$ is $0_{a,b}$, then $(M_u)_{i,j}=\pi_b^{i}\circ 0_{a,b}\circ \iota_a^j=0_{a^j, b^{i}}$.
\end{proof}

\begin{cor}\label{4.3}
Every entry of $M_u$ is expressible without using $\oplus$, $\iota$ and $\pi$.
\end{cor}
\begin{proof}
Since $u$ is built out of terms of the form $\mj_a$, $\alpha_{a,b,c}$, $\lambda_a$, $\sigma_{a,b}$, $\eta_{a,b}$, $\varepsilon_{a,b}$, $\iota^i_{a,b}$, $\pi^i_{a,b}$ and $0_{a,b}$ with the help of $\otimes$, $a\multimap$, $\oplus$, $+$ and $\circ$, one has just to apply Propositions~\ref{matrix1} and \ref{matrix2}.
\end{proof}

\begin{cor}\label{4.4}
Every arrow of $\mathcal{F}_P$ whose source and target are $\oplus$-free is expressible without using $\oplus$, $\iota$ and $\pi$.
\end{cor}
\begin{proof}
If the source and the target of $u$ are $\oplus$-free, then by Remark~\ref{3.1}, the only entry of $M_u$ is $u$ itself and it remains to apply Corollary~\ref{4.3}.
\end{proof}

\section{The graphical language}

A special SMCB category, which serves as a model (or a graphical language) for the arrows of $\mathcal{F}_P$ is introduced in this section. The essential ingredient of this category is the category $1\cob$ described in Example~\ref{ex7}.

Let $1\cob^+$ be the category with the same objects as $1\cob$, while the arrows of $1\cob^+$ from $a$ to $b$ are the finite (possibly empty) multisets of arrows of $1\cob$ from $a$ to $b$. (We abuse the notation by using the set brackets $\{$, $\}$ for multisets.) The identity arrow $\mj_a\colon a\to a$ is the singleton multiset $\{\mj_a\colon a\to a\}$, while the composition of $\{f_j\colon a\to b\mid j\in J\}$ and $\{f_k\colon b\to c\mid k\in K\}$ is
\[
\{f_k\circ f_j\colon a\to c\mid j\in J, k\in K\}.
\]
The category $1\cob^+$ is enriched over the category \textbf{Cmd}. The addition on hom-sets is the operation + (disjoint union) on multisets and the neutral is the empty multiset.

Let $1\cob^\oplus$ be the biproduct completion of $1\cob^+$ constructed as in \cite[Section~5.1]{S07}. The objects of $1\cob^\oplus$ are the finite sequences $\langle a_0,\ldots,a_{n-1}\rangle$, $n\geq 0$, of objects $a_0,\ldots,a_{n-1}$ of $1\cob$. For example, $\langle++-+--,+,--+\rangle$ is an object of $1\cob^\oplus$. Note the distinction between the empty sequence $\emptyset$ and the sequence $\langle\emptyset\rangle$ whose only member is the empty sequence of oriented points. The arrows of $1\cob^\oplus$ from $\langle a_0,\ldots,a_{n-1}\rangle$ to $\langle b_0,\ldots,b_{m-1}\rangle$ are the $m\times n$ matrices whose $ij$ entry is an arrow of $1\cob^+$ from $a_j$ to $b_i$. This category has the role of graphical language for symmetric monoidal closed categories with biproducts.

\begin{rem}
The commutativity of diagrams is decidable in $1\cob^\oplus$.
\end{rem}

\begin{prop}\label{dagger}
The category $1\cob^\oplus$ is dagger compact closed with dagger bi\-products.
\end{prop}
\begin{proof}
The category $1\cob^{+}$ is dagger compact closed. For arrows $f$ and $g$ of $1\cob^{+}$ given by the multisets $\{f_i:a\to b\mid i\in I\}$ and $\{g_j:c\to d\mid j\in J\}$ respectively, we define $f\otimes g$ as $\{f_{i}\otimes g_j\mid i\in I, j\in J\}$. Similarly, $f^{\dagger}$ is defined as $\{f_i^\dagger:b\to a\mid i\in I\}$. This category is enriched over $\cmd$ as a compact closed category, and it is straightforward to check that, for $f,f'\colon a\to b$, we have $(f+f')^\dagger=f^\dagger+f'^\dagger$ and $0^\dagger=0$. Now the claim follows from \cite[Proposition~5.1]{S07}.
\end{proof}

\begin{cor}
The category $1\cob^\oplus$ is an SMCB category.
\end{cor}

\section{Coherence}\label{scoh}

This section contains the main result of our paper. We start with some auxiliary notions. The $I$-\emph{valued} and $0$-\emph{valued} objects of $\mathcal{F}_P$ are inductively defined as follows.
\begin{itemize}
\item[(1)] $I$ is $I$-valued and $0$ is $0$-valued;
\item[(2)] $a\oplus b$ is $I$-valued when one of $a$ and $b$ is $I$-valued and the other is $0$-valued, and $a\oplus b$ is $0$-valued when both are $0$-valued;
\item[(3)] $a\otimes b$ ($a\multimap b$) is $I$-valued when both $a$ and $b$ are $I$-valued, and $a\otimes b$ ($a\multimap b$) is $0$-valued when at least one of $a$ and $b$ is $0$-valued.
\end{itemize}

\begin{rem}
If $a$ is $\oplus$-free and $I$-valued, then $a$ is built out of $I$, $\otimes$ and $\multimap$, only.
\end{rem}

An object $a$ of $\mathcal{F}_P$ is \emph{proper} when for every subformula of $a$ of the form $b\multimap c$, if $c$ is $I$-valued, then $b$ is either $I$-valued or $0$-valued. An object $a$ of $\mathcal{F}_P$ is $I$-\emph{proper} when for every subformula of $a$ of the form $b\multimap c$, if $c$ is $I$-valued, then $b$ is $I$-valued.

\begin{rem}\label{iota-pi proper}
By the definition of $\iota^i_a$ and $\pi^i_a$, we have that if $a$ is proper, then the source of $\iota^i_a$ (the target of $\pi^i_a$) is proper too.
\end{rem}

\begin{rem}\label{6.2}
If $a$ is proper and $\oplus$-free, then either it is a zero object, or it contains no $0$ and is $I$-proper.
\end{rem}

Consider the function $g$ from the set $P$ of generators of $\mathcal{F}_P$ to the objects of $1\cob^\oplus$ that maps every element of $P$ to the singleton sequence $\langle +\rangle$. Since $\mathcal{F}_P$ is a SMCB category freely generated by the set $P$, and $1\cob^\oplus$ is a SMCB category, there exists a unique SMCB functor (one that strictly preserves the SMCB structure) $G\colon \mathcal{F}_P\to 1\cob^\oplus$, which extends the function $g$. We call an arrow of $\mathcal{F}_P$, which is expressed in \emph{pure symmetric monoidal closed language} (free of $\oplus$,  $+$ and $0,\iota,\pi$-arrows) an SMC-\emph{arrow}. Note that if $u$ is an SMC-arrow, then $Gu$ corresponds to the Kelly-Mac Lane graph of $u$.

For the proof of Theorem~\ref{coh} below, we use the following version of the partial coherence theorem for symmetric monoidal closed categories proved by Kelly and Mac Lane \cite[Theorem~2.4]{KML71} (see also \cite[Section 1.1, second paragraph]{S97}).

\begin{thm}[SMC Coherence]\label{smccoh}
If $a$ and $b$ are $I$-proper and $f,g\colon a\to b$ are SMC-arrows such that $Gf=Gg$, then $f=g$.
\end{thm}

The following theorem is the main result of the paper.

\begin{thm}[SMCB Coherence]\label{coh}
If $a$ and $b$ are proper and $f,g\colon a\to b$ are arrows of $\mathcal{F}_P$ such that $Gf=Gg$, then $f=g$, i.e.\ the restriction of $G$ to the full subcategory of $\mathcal{F}_P$ on the set of proper objects is faithful.
\end{thm}
\begin{proof}
Let us show that for every $\iota^i_a\in {\rm I}_a$ and  $\pi^j_b\in\Pi_b$, we have that $\pi^j_b\circ f\circ\iota^i_a=\pi^j_b\circ g\circ\iota^i_a$. By Corollary~\ref{4.4}, with the help of equalities~\ref{17}-\ref{18} and the equalities listed at the end of Section~\ref{sec3}, it follows that $\pi^j_b\circ f\circ\iota^i_a$ is either equal to $0_{a^i,b^j}$, or to $\sum_{k=1}^n f_k$, $n\geq 1$, where every $f_k$ is an SMC-arrow. By the same reasons, $\pi^j_b\circ g\circ\iota^i_a$ is either equal to $0_{a^i,b^j}$, or to $\sum_{k=1}^m g_k$, $m\geq 1$, where every $g_k$ is an SMC-arrow.

If $\pi^j_b\circ f\circ\iota^i_a=0_{a^i,b^j}$, then $G(\pi^j_b\circ f\circ\iota^i_a)$ is the empty multiset. From $Gf=Gg$ we conclude that $G(\pi^j_b\circ g\circ\iota^i_a)$ must be the empty multiset too, and since $G(\sum_{k=1}^m g_k)$, for $m\geq 1$, cannot be such, it follows that $\pi^j_b\circ g\circ\iota^i_a=0_{a^i,b^j}$. We proceed analogously when $\pi^j_b\circ g\circ\iota^i_a=0_{a^i,b^j}$.

If $\pi^j_b\circ f\circ\iota^i_a=\sum_{k=1}^n f_k$ and $\pi^j_b\circ g\circ\iota^i_a=\sum_{k=1}^m g_k$, for $n,m\geq 1$, where $f_k$ and $g_k$ are SMC-arrows, then from $Gf=Gg$, it follows that
\[
\{Gf_k\mid 1\leq k\leq n\}=G\sum_{k=1}^n f_k=G\sum_{k=1}^m g_k= \{Gg_k\mid 1\leq k\leq m\}.
\]
Hence, the multisets $\{Gf_k\mid 1\leq k\leq n\}$ and $\{Gg_k\mid 1\leq k\leq m\}$ have the same number of elements, i.e., $n=m$, and, without loss of generality, we may conclude that for every $1\leq k\leq n$, $Gf_k=Gg_k$. By Remark~\ref{iota-pi proper}, we have that the source and the target of $f_k$ and $g_k$ are proper, and since $f_k$ and $g_k$ are SMC-arrows, by Remark~\ref{6.2}, they are $I$-proper. From Theorem~\ref{smccoh}, it follows that for every $1\leq k\leq n$, $f_k=g_k$, and hence, $\pi^j_b\circ f\circ\iota^i_a=\pi^j_b\circ g\circ\iota^i_a$.

Since the above holds for every $\iota^i_a\in {\rm I}_a$ and $\pi^j_b\in\Pi_b$, by Corollary~\ref{3.3}, we conclude that $f=g$.
\end{proof}

\section{The case of compact closed categories with biproducts}

Along the lines of our proof of Theorem~\ref{coh}, one can prove an analogous result concerning compact closed categories with biproducts (CCB categories). The appropriate CCB language is obtained from the SMCB language as follows. The binary operation $\multimap$ on objects is replaced by the unary operation $^\ast$. The unary operations $a\multimap$ on arrows are omitted. The families of arrows with components $\eta_{a,b}$ and $\varepsilon_{a,b}$ are replaced by the families of arrows with components $\eta_a\colon I\to a^\ast\otimes a$ and $\varepsilon_a\colon a\otimes a^\ast\to I$. Eventually, the equalities \ref{4}, \ref{8}, \ref{9} and \ref{12} should be replaced by the equalities \ref{triang}.

Let $\mathcal{F}'_P$ be the CCB category freely generated by a set $P$, constructed with respect to the above language. For the function $g$ from $P$ to the set of objects of $1\cob^\oplus$ defined as in Section~\ref{scoh}, there is a unique CCB functor $G\colon \mathcal{F}'_P\to 1\cob^\oplus$, which extends the function $g$. We call an arrow of $\mathcal{F}'_P$, which is expressed in \emph{pure compact closed language} (free of $\oplus$,  $+$ and $0,\iota,\pi$-arrows) an CoC-\emph{arrow}. Note that if $u$ is a CoC-arrow, then $Gu$ corresponds to the Kelly-Laplaza graph of $u$ (see \cite{KL80}).

In order to prove a coherence result for CCB categories we need to introduce some auxiliary notions and to modify definitions and results given in Sections \ref{free SMCB} and~\ref{matrix}. The contravariant functor $^\ast$ is defined in the standard way---for $f\colon a\to b$,
\[
f^\ast=\lambda_{a^\ast}\circ \sigma_{a^\ast,I}\circ (a^\ast\otimes \varepsilon_b) \circ \alpha^{-1}_{a^\ast,b,b^\ast}\circ((a^\ast\otimes f)\otimes b^\ast)\circ(\eta_a\otimes b^\ast)\circ\lambda_{b^\ast}^{-1}.
\]
For every object $a$ of $\mathcal{F}'_P$, one can define the sequences  ${\rm I}_a$ and $\Pi_a$, by replacing the item $\multimap$ in Definition~\ref{proj-inj} by
\begin{itemize}
\item[$\ast$] If $a=a_1^\ast$, then $n=n_1$, and for $0\leq i< n_1$,
    $\iota^i_a=(\pi^i_1)^\ast$, $\pi^i_a=(\iota^i_1)^\ast$.
\end{itemize}
It is straightforward to check that Proposition~\ref{4.2} and Corollary~\ref{3.3}, with $\mathcal{F}_P$ replaced by $\mathcal{F}'_P$, remain to hold.

By omitting the case (2) in the proof of Proposition~\ref{matrix1} we obtain the analogous proposition for $\mathcal{F}'_P$.
\begin{prop}\label{matrix1'}
For $\bullet$ being $\otimes$, $\oplus$, $+$ and $\circ$, we have
\[
M_{u_1\bullet u_2}=M_{u_1}\bullet M_{u_2}.
\]
\end{prop}

The following proposition is analogous to Proposition~\ref{matrix2}.

\begin{prop}\label{matrix3}
If $u$ is of the form $\mj_a$, $\alpha_{a,b,c}$, $\lambda_a$, $\sigma_{a,b}$, $\eta_a$, $\varepsilon_a$, $\iota^i_{a,b}$, $\pi^i_{a,b}$ or $0_{a,b}$, then all the entries of the matrix $M_u$ are of the form $\mj_p$, $\alpha_{p,q,r}$, $\lambda_p$, $\sigma_{p,q}$, $\eta_p$, $\varepsilon_p$ and $0_{p,q}$, where $p$ and $q$ are $\oplus$-free.
\end{prop}
\begin{proof}
Just replace the cases (4) and (5) in the proof of Proposition~\ref{matrix2} by $(4')$ and $(5')$ below.

\medskip
\noindent $(4')$ If $u$ is $\eta_a$, then the matrix $M_u$ is a vector column and for some $i_1$, $i_2$ we have
\begin{align*}
(M_u)_{i,1} &= \pi_{a^\ast\otimes a}^i\circ \eta_a=((\iota_a^{i_1})^\ast\otimes \pi_a^{i_2}) \circ \eta_a \stackrel{\clubsuit}{=} ((a^{i_1})^\ast\otimes (\pi_a^{i_2}\circ \iota_a^{i_1}))\circ  \eta_{a^{i_1}} \\
&=\begin{cases}
\eta_{a^{i_1}}, & i_1=i_2, \\
0_{I, (a^{i_1})^\ast\otimes a^{i_2}}, & \text{otherwise}.
\end{cases}
\end{align*}
Note that $\clubsuit$ holds since for $f\colon a\to b$, we have
$(f^\ast\otimes b)\circ\eta_b=(a^\ast\otimes f)\circ \eta_a$,
which is derived essentially by the right hand side of \ref{triang}, with the help of \ref{2}, and symmetric monoidal coherence (see \cite[Section~5]{ML63}, \cite[XI.1, Theorem~1]{ML71} and \cite[Section~5.3]{DP04}).

\medskip
\noindent $(5')$ We proceed analogously when $u$ is $\varepsilon_a$.
\end{proof}

The following result is related to \cite[Theorem~21]{AD04} but it is formulated and proved in a different manner. As usually (cf.\ \cite[Theorem~3.1]{ML63} and \cite[Proposition~3]{S63}) it is difficult to give a full comparison of these two coherence results, even one may consider them as having the same ``mathematical content''.

\begin{thm}[CCB Coherence]\label{coh1}
The functor $G\colon\mathcal{F}'_P\to 1\cob^\oplus$ is faithful.
\end{thm}
\begin{proof}
We start with $f,g\colon a\to b$ such that $Gf=Gg$ and proceed as in the proof of Theorem~\ref{coh} until we get that $\pi^j_b\circ f\circ\iota^i_a=\sum_{k=1}^n f_k$ and $\pi^j_b\circ g\circ\iota^i_a=\sum_{k=1}^n g_k$, for $n\geq 1$, where $f_k$ and $g_k$ are CoC-arrows and for every $1\leq k\leq n$, $Gf_k=Gg_k$. By relying on \cite[Theorem~8.2]{KL80}, we conclude that for every $1\leq k\leq n$, $f_k=g_k$, and hence, $\pi^j_b\circ f\circ\iota^i_a=\pi^j_b\circ g\circ\iota^i_a$. It remains to apply Corollary~\ref{3.3} with $\mathcal{F}_P$ replaced by $\mathcal{F}'_P$.
\end{proof}

Concerning the case of dagger compact closed categories with dagger biproducts (DCCB categories), the appropriate language is obtained from the CCB language by the following modifications. A unary operation $^\dagger$ on arrows is added. The families of arrows $\alpha^{-1}$, $\lambda^{-1}$, $\eta$ and $\iota$ are omitted. The equalities $f^{\dagger\dagger}=f$, \ref{a}, \ref{b} (the third one) and \ref{c} are added, and the arrows $\alpha^{-1}_{a,b,c}$, $\lambda^{-1}_a$, $\eta_a$, $\iota^i_{a,b}$ are replaced by $\alpha_{a,b,c}^\dagger$, $\lambda_a^\dagger$, $\sigma_{a,a^\ast}\circ \varepsilon_a^\dagger$, $(\pi^i_{a,b})^\dagger$ in the equalities assumed for CCB categories.

Let $\mathcal{F}''_P$ be the DCCB category freely generated by a set $P$. Since it is also a CCB category, there is a unique CCB functor $G'\colon \mathcal{F}'_P\to \mathcal{F}''_P$, which extends the identity function on $P$. This functor is an isomorphism that is identity on objects. On the other hand, there is a unique DCCB functor $G''\colon \mathcal{F}''_P\to 1\cob^\oplus$, which extends the function $g$. By the uniqueness, the above $G$ is equal to the composition $G''\circ G'$, and since $G$ is faithful, and $G'$ is an isomorphism, we have the following result.

\begin{thm}[DCCB Coherence]\label{coh2}
The functor $G''\colon\mathcal{F}''_P\to 1\cob^\oplus$ is faithful.
\end{thm}

\section{Switching between graphical languages}

This section serves just as a sketch of our programme for a future work. It contains no precisely formulated results and is far from being self-contained.

The graphical language for symmetric monoidal closed categories consists of Kelly-Mac Lane graphs, i.e.\ the arrows of $1\cob$. On the other hand, an appropriate graphical language for biproducts is the one given in \cite[Section~6.3]{S09}, which may be formalised through the category $\mat_{\mathbb{N}}$ (see Example~\ref{ex6}). These two graphical languages do not cooperate well, as it was noted in \cite[Section~3, last paragraph]{S07}. Our solution of SMCB coherence relies on a construction based on the category $1\cob$ and it is reasonable to ask whether this coherence could be obtained by relying on the category $\mat_{\mathbb{N}}$ instead.

One way to switch from the graphical language based on $1\cob$ to the one based on $\mat_{\mathbb{N}}$ is to use 1-dimensional topological quantum field theories, which are all (with minor provisos) faithful according to \cite{T19}. In particular, the proof of the main result of \cite{DKP06} could be modified in order to show that there is a faithful functor from a monoidal closed category (without symmetry) with biproducts freely generated by a set of objects to the category $\mat_{\mathbb{N}}$. This functor strictly preserves the structure of monoidal closed categories with biproducts. In order to construct such a functor, we start with one defined as $G\colon \mathcal{F}_P\to 1\cob^\oplus$ in Section~\ref{scoh}, save that now its source is a monoidal closed category with biproducts freely generated by a set of objects. By composing such $G$ with a functor obtained as a modification of Brauer's representation of Brauer's algebras (see \cite{B37}, \cite{W88}, \cite{J94} and \cite{DP12}) one obtains the desired faithful functor. The existence of such a functor in presence of symmetry is still an open problem for us.

\begin{center}\textmd{\textbf{Acknowledgements} }
\end{center}
\medskip
This work was supported by the Serbian Ministry of Education, Science and Technological Development through Mathematical Institute of the Serbian Academy of Sciences and Arts. The authors are grateful to the editor and to the anonymous referee for careful reading and valuable comments and remarks, which improved the presentation of the paper.

\end{document}